\documentclass[12pt]{amsart}

\usepackage{graphicx}

\usepackage[small,heads=vee]{diagrams}
\diagramstyle[labelstyle=\scriptstyle]

\def\C{\mathbb{C}}
\def\S{\mathbb{S}}

\def\H{\mathbb{H}}
\def\P{\mathbb{P}}

\def\cC{\mathcal{C}}
\def\cD{\mathcal{D}}

\def\cQ{\mathcal{Q}}

\def\al{\alpha}

\def\ga{\gamma}
\def\de{\delta}

\def\si{\sigma}

\def\De{\Delta}

\def\Om{\Omega}

\newcommand{\der}{{\rm d}}

\numberwithin{equation}{section}

\newtheorem{theorem}{Theorem}[section]

\newtheorem{proposition}{Proposition}[section]
\newtheorem{corollary}{Corollary}[section]

\newtheorem{Definition}{Definition}[section]

\theoremstyle{remark}

\usepackage{amssymb,stmaryrd}
\usepackage{amscd}

\author{Matthew Randall}
\address{Department of Mathematics and Statistics\\
Faculty of Science, Masaryk University\\
Kotl\'a\v{r}sk\'a 2, 611 37 Brno\\
Czech Republic}
\email{randallm@math.muni.cz}

\title{Schwarz triangle functions and duality for certain parameters of the generalised Chazy equation}

\thanks{This work is supported by the Grant agency of the Czech Republic P201/12/G028.}

\begin{document}

\begin{abstract}
Schwarz triangle functions play a fundamental role in the solutions of the generalised Chazy equation. Chazy has shown that for the parameters $k=2$ and $3$, the equations can be linearised. We determine the Schwarz triangle functions that represent the solutions in these cases. Some of these Schwarz triangle functions also show up in the dual cases where $k=\frac{2}{3}$ and $k=\frac{3}{2}$, which suggests intriguing connections between the solutions for $k=2$ and $k=3$ with dihedral and tetrahedral symmetry respectively, and $k=\frac{2}{3}, \frac{3}{2}$ with $G_2$ symmetry.
Integrating the solutions when $k=\frac{2}{3}$, we obtain flat $(2,3,5)$-distributions parametrised algebraically by the corresponding Schwarz triangle functions. The Legendre dual of these curves are also algebraic, can be computed quite easily and are related to the integral curves of the dual generalised Chazy equation with parameter $\frac{3}{2}$. 
\end{abstract}

\maketitle

\pagestyle{myheadings}
\markboth{Randall}{Schwarz triangle functions and duality for certain parameters of the generalised Chazy equation}

The generalised Chazy equation is a third order nonlinear autonomous ordinary differential equation given by
\begin{equation}\label{gc0}
y'''-2 y'' y+3 (y')^2-\frac{4}{36-k^2}(6y'-y^2)^2=0
\end{equation}
for $k \neq 6$. The equation (\ref{gc0}) was introduced by Jean Chazy in the papers \cite{chazy1} and \cite{chazy2} in the context of investigating the Painlev\'e property for third order ODEs. 

The equation (\ref{gc0}) can be solved by Schwarz triangle functions \cite{acht}. Schwarz triangle functions determine through its inverse a map from the complex upper half plane to an open triangular domain with boundary given by the edges of the triangle. The angles of the triangle determined by the Schwarz functions depend on the parameter $k$ in (\ref{gc0}).
When $k<6$, the image is a spherical triangle. When $k=6$, the triangle is planar and when $k>6$, the triangle is hyperbolic.  In this article we present the spherical Schwarz triangle functions corresponding to the solutions when $k$ is given by $\frac{2}{3}$, $\frac{3}{2}$, $2$ and $3$ and determine them algebraically.

These four parameters are chosen because the equations also show up in the context of the geometry of differential equations. The problem of determining whether the solution set of a system of differential equation is equivalent to another, via for instance point or contact preserving transformation, can be solved using Cartan's method of equivalence. 
 
The equations for the parameters $k=2$ and $k=3$ are shown to be linearisable by Chazy himself (see p. 346 of \cite{chazy2}). The equation when $k=2$ is linearisable to the ODE $y''''=0$. This is related to the third order Riccati equation as observed in \cite{guha}. Applying Cartan's method of equivalence, this 4th order ODE has vanishing Wilczynski invariants in the linear theory and also vanishing Bryant invariants in the non-linear theory \cite{bryant}, \cite{doubrov}, \cite{gp}, \cite{gl2r}. We discuss the general solution when $k=2$ in Section \ref{gen2}.

The method of equivalence also applies to third order ODEs, as worked out by Chern \cite{chern}. The equation for the parameter $k=3$ turns out to be the only equation of the form (\ref{gc0}) that has vanishing W\"unschmann invariant. Third order ODEs with vanishing W\"unschmann invariant defines a conformal structure of signature $(2,1)$ on the space of its solutions. The conformal metric is obtained by quotienting out a degenerate split signature symmetric 2-tensor by the vector field that annihilates the distribution encoding the ODE $y'''=F(x,y,p,q)$. ODEs with vanishing W\"unschmann invariant and $F_{qqqq}=0$, satisfied for the $k=3$ equation, are contact equivalent to the equation $y'''=0$ \cite{gn}. The equation for this parameter $k=3$ is also linearisable and the general solution to this equation is described in Section \ref{gen3}.

The equations for the parameters $k=\pm \frac{3}{2}$ and $k=\pm \frac{2}{3}$ show up in the local equivalence problem for generic maximally non-integrable (or bracket generating) rank $2$ distributions on $5$-manifolds $M$ that depend on a single function $F(x)$. Here the genericity condition implies that $F''(x) \neq 0$. For such non-integrable distributions, the bracket of the vector fields spanning the distribution $\cD$ determines a filtration of the tangent bundle given by $\cD \subset [\cD,\cD]\subset TM=[[\cD,\cD],\cD]$ with the rank of $\cD=2$ and the rank of $[\cD,\cD]=3$. Such distributions are therefore also known as $(2,3,5)$-distributions. Cartan solved the local equivalence problem for such geometries in \cite{cartan1910} and constructed the fundamental curvature invariant. For distributions of the form $F(x)$ as described, the curvature invariant vanishes when $F''(x)=e^{\frac{1}{2}\int y(x) \der x}$ where $y(x)$ is a solution to the generalised Chazy equation (\ref{gc0}) with parameter $k=\pm \frac{2}{3}$. In this case the distribution has split $G_2$ as its local group of symmetries. 
Furthermore An and Nurowski \cite{annur} showed that there is a duality that takes the solutions of this equation to the solutions of the submaximal $7^{\rm th}$ order ODE 
\begin{equation}\label{noth}
10(y''')^3y^{(7)}-70(y''')^2y^{(4)}y^{(6)}-49(y''')^2(y^{(5)})^2+
280y'''(y^{(4)})^2y^{(5)}-175(y^{(4)})^4=0
\end{equation}
studied in \cite{olver} (where it appears in equation (6.64)), \cite{doubrov} and \cite{ds}. Historically, this ODE appeared already in the thesis of Noth \cite{noth} in 1904. 

This dual ODE (\ref{noth}) can also be reduced to a generalised Chazy equation but now the Chazy parameter is given by $k=\pm \frac{3}{2}$. In this fashion, the solutions with parameters $k=\frac{3}{2}$ and $k=\frac{2}{3}$ have vanishing Cartan invariants. We discuss the solutions for both these equations in Sections \ref{gen23} and \ref{gen32}. We show in Appendix \ref{uod} that the Legendre duality property for the generalized Chazy equations with parameters $k=\pm \frac{3}{2}$ and $k=\pm \frac{2}{3}$ is unique only for these parameters.

Interestingly, one of the Schwarz triangle functions that solve the $k=2$ equation also show up in the solutions to the $k=\frac{2}{3}$ equation.
Three of the Schwarz triangle functions that solve the $k=3$ equation also show up in the solutions to the $k=\frac{3}{2}$ equation. For the $k=2$ and $k=\frac{2}{3}$ cases, the Schwarz triangle functions are pullbacks through hypergeometric transformation of the Schwarz function $s(\frac{1}{2},\frac{1}{3},\frac{1}{2},x)$ that appears in Schwarz's list \cite{schwarz} with dihedral symmetry. For the $k=3$ and $k=\frac{3}{2}$ cases, the functions are pullbacks of the function $s(\frac{1}{2},\frac{1}{3},\frac{1}{3},x)$ that appears in Schwarz's list with tetrahedral symmetry. The result of Schwarz \cite{schwarz} and Klein tells us that the Schwarz functions we obtain in these cases are algebraic. The maps for the triangle functions that show up in the cases $k=\frac{2}{3}$, $\frac{3}{2}$, $2$ and $3$ are given in the Diagrams \ref{map1}, \ref{map2} and \ref{map3}. They show the hypergeometric transformations that are given by quadratic, cubic and quartic maps \cite{goursat}, \cite{hyper}.

Another motivation for the article is to work out the Schwarz functions that solve the $k=\frac{2}{3}$ and $k=\frac{3}{2}$ equation and determine examples of $(2,3,5)$-distributions with maximal symmetry group of split $G_2$. The solutions to the $k=\frac{2}{3}$ equation has appeared in \cite{r16}, but we present them more explicitly here in the form of Table \ref{table2}. We also present the solutions when $k=\frac{3}{2}$ here. To work out the distributions $D_{F(x)}$ with vanishing Cartan curvature invariant, we have to determine $F''(x)$ from the solutions of the $k=\frac{2}{3}$ Chazy equation and integrate twice further. This gives an algebraic relation involving $(x,F)$. The dual curve of this plane algebraic curve gives us integral curves of equation (\ref{noth}).

The first two sections are background material. In Section \ref{sl2-section} we set up the preliminaries and consider $SL_2(\C)$ equivalent classes of solutions to (\ref{gc0}). In Section \ref{schwarz-section} we review the definitions of Schwarz functions and in Sections \ref{gen2}, \ref{gen23}, \ref{gen3}, \ref{gen32} we present the Schwarz triangle functions that appear in the solutions of the Chazy's equation for $k=2$, $k=\frac{2}{3}$, $k=3$ and $k=\frac{3}{2}$ respectively.
The computations are done through MAPLE 17. 

\tableofcontents

\section{$SL_2(\C)$ action on the space of solutions}\label{sl2-section}

The material in the first two sections is collated from \cite{chazy2}, \cite{co96}, \cite{cosgrove}, \cite{ach} and \cite{acht}. We shall work over the complex field. 
 Any element $g={\bigl(\begin{smallmatrix}a & b\\ c& d \end{smallmatrix}}\bigr)\in SL_2(\C)$ acts on $x$ by fractional linear transformations $g \cdot x=\tilde x=\frac{ax+b}{cx+d}$.
 \begin{proposition} (See also \cite{co96}, \cite{ach})
Under this action of $SL_2(\C)$, for any solution $y(x)$ to the generalised Chazy equation (\ref{gc0}) with $k \neq 6$, we obtain new solutions to (\ref{gc0}) by
 \begin{align}\label{sl2}
 \tilde y(x)=\frac{1}{(cx+d)^2}y\left(\frac{ax+b}{cx+d}\right)-\frac{6c}{cx+d}.
 \end{align}
 \end{proposition}
 \begin{proof}
 The action of $SL_2$ gives the differential relation $\der \tilde x=\frac{1}{(cx+d)^2}\der x$ and $\frac{\der}{\der \tilde x}=(cx+d)^2 \frac{\der}{\der x}$. Differentiating, we find
  \begin{align*}
(c x+d)^{4}(6\tilde y'-\tilde y^2)= 6 y' -y^2
 \end{align*}
 and 
 \begin{align*}
(c x+d)^{6}(9 \tilde y''-9 \tilde y\tilde y'+\tilde y^3)=9 y''-9 y y'+y^3,
 \end{align*}
 where prime on the left hand side denotes differentiation with respect to $x$ while prime on the right hand side denotes differentiation with respect to $\tilde x.$
Differentiating once more gives
\begin{align*}
&(c x+d)^{8}\left(\tilde y'''-2\tilde y'' \tilde y +3(\tilde y')^2-\frac{4}{36-k^2}(6\tilde y'-\tilde y^2)^2\right)\\
=&y'''-2 y'' y+3 (y')^2-\frac{4}{36-k^2} (6 y' -y^2)^2.
 \end{align*}
 We see that $\tilde y(x)$ is a solution to (\ref{gc0}) iff $y(\tilde x)$ is a solution as well. 
\end{proof}

Let $f(x)=\exp(\frac{2}{k-6} \int y \der x)$. 
Chazy makes this substitution and finds that $f$ satisfies the differential equation
\begin{align}\label{intc}
ff''''-(k-2)f' f'''+\frac{3k (k-2)}{2(k+6)}(f'')^2=0.
\end{align}
It is immediate from (\ref{intc}) that when $k=2$, the equation becomes linear, and we shall discuss this further in Section \ref{gen2}. When we integrate (\ref{sl2}), we obtain 
 \begin{align*}
 \int \tilde y(x) \der x=& \int \frac{1}{(c x+d)^2} y(\tilde x) \der x-\int \frac{6c}{c x+d}\der x\\
 =&\int y(\tilde x) \der \tilde x-6 \log( cx+d)+c_0.
 \end{align*}
We find that 
\begin{align*}
\tilde f(x)=\exp\left({\frac{2}{k-6}\int \tilde y(x) \der x}\right)
 =&\frac{\exp({\frac{2c_0}{k-6}})}{(c x+d)^{\frac{12}{k-6}}}\exp\left({\frac{2}{k-6}\int y(\tilde x) \der \tilde x}\right)=\frac{\exp(\frac{2c_0}{k-6})}{(c x+d)^{\frac{12}{k-6}}}f(\tilde x).
\end{align*} 
Absorbing constants (or normalising them so that $c_0=0$), we see that
\begin{align*}
\tilde f(x)=\frac{1}{(c x+d)^\frac{12}{k-6}}f(\tilde x).
\end{align*} 
This motivate the following definition. 

\begin{Definition}\label{def2}
Suppose both functions $f(x)$ and $\tilde f(x)=(cx+d)^{-\frac{12}{k-6}}f(\tilde x)$ satisfy the same differential equation (\ref{intc}). Then we say that the function $f(x)$ has {\bf weight} $\frac{12}{k-6}$ since $f(\tilde x)=(cx+d)^{\frac{12}{k-6}}\tilde f(x)$  (following the convention in the literature about weights of modular forms).
\end{Definition}

Let us take $k=\frac{2}{3}$ and suppose $f(x)=(F''(x))^{-\frac{3}{4}}$ for some $F(x)$. Then $F''(x)$ has weight $3$ under the action of $SL_2(\C)$ and we find that $F(x)$ satisfies the 6th order ODE 
\begin{align}\label{6ode}
10F^{(6)}(F'')^3-80(F'')^2F^{(3)}F^{(5)}-51(F'')^2(F^{(4)})^2
+336F''(F''')^2F^{(4)}-224(F''')^4=0
\end{align}
in \cite{annur} upon substituting $f=(F'')^{-\frac{3}{4}}$ into (\ref{intc}). 
\begin{proposition}
If $F(x)$ is a solution to the 6th order ODE (\ref{6ode}), then so is 
\[
\tilde F(x)=(cx+d)F\left(\frac{ax+b}{cx+d}\right)
\]
where $ad-bc=1$.
\end{proposition}
According to the definition given above, the function $F(x)$ has weight $-1$. We have
\begin{corollary}
The function $F(x)=x^2$ is a solution to the 6th order ODE (\ref{6ode}), and therefore so is 
\[
\tilde F(x)=(cx+d)\left(\frac{ax+b}{cx+d}\right)^2=\frac{(ax+b)^2}{cx+d}
\]
where $ad-bc=1$.
\end{corollary}
Differentiating the above relation twice, we find that
\begin{align*}
\tilde F''(x)=\frac{1}{(cx+d)^3}F''(\tilde x).
\end{align*}
Again the right hand side denotes differentiation with respect to $\tilde x$. For any $F(x)$ with weight $-1$ satisfying the 6th order ODE (\ref{6ode}), we identify the solutions
 \[
 \left(x,\tilde F(x)\right)=\left(x,(cx+d)F(g \cdot x)\right)\sim\left(g \cdot x,F(g\cdot x)\right).
 \]
 
We define $SL_2(\C)$ equivalent solutions to the generalised Chazy's equation in the following fashion (see also \cite{co96}).
\begin{Definition}\label{def1}
Two solutions $y(\tilde x)$ and $\tilde y(x)$ to the generalised Chazy equation are said to be {\bf{equivalent}} if there exists an element $g$ of $SL_2(\C)$ such that $\tilde x=g\cdot x$ and (\ref{sl2}) holds for $y(g\cdot x)$ and $\tilde y(x)$. 
\end{Definition}
From this we can identify the solutions
 \[
 \left(x,\tilde y(x)\right)\sim\left(\tilde x,(cx+d)^2\tilde y(x)+6c(cx+d)\right)=(\tilde x, y(\tilde x)).
 \]
It is well known that $F(x)=x^m$ for $m \in \{-1,\frac{1}{3},\frac{2}{3},2\}$ solves equation (\ref{6ode}) (see \cite{cartan1910}). Let us investigate how these solutions are equivalent solutions under $SL_2(\C)$.
Using that $y(x)=2\frac{\der}{\der x}\log(F''(x))$, this gives $y(x)=\frac{6}{x}$, $y(x)=-\frac{10}{3x}$, $y(x)=-\frac{8}{3x}$ and $y(x)=0$ as solutions to the $k=\frac{2}{3}$ equation. The solution $y=-\frac{8}{3(x+C)}-\frac{10}{3(x+B)}$ for constants $B$, $C$ was further obtained following \cite{chazy1} and \cite{chazy2}. This corresponds to $F(x)=(x+B)^{\frac{1}{3}}(x+C)^{\frac{2}{3}}$.

\begin{proposition}\label{001}
The solutions to (\ref{gc0}) for the parameter $k=\frac{2}{3}$ given by
$y=-\frac{10}{3x}$, $y=-\frac{8}{3x}$ and $y=-\frac{8}{3(x+C)}-\frac{10}{3(x+B)}$ are equivalent in the sense of Definition \ref{def1}.
\end{proposition}
\begin{proof}
Applying an arbitrary $g=\bigl(\begin{smallmatrix}a & b\\ c& d \end{smallmatrix}\bigr) \in SL_2(\C)$ to the solution given by $y=-\frac{8}{3x}$, we obtain  
$\tilde y=-\frac{8}{3(x+\frac{b}{a})}-\frac{10}{3(x+\frac{d}{c})}$. Applying $g_1=\Bigl(\begin{smallmatrix} -e c & \frac{f}{c (e-f)}\\ c& -\frac{1}{e c}(1+ \frac{f}{e-f}) \end{smallmatrix}\Bigr) \in SL_2(\C)$ to $\tilde y=-\frac{8}{3(x+e)}-\frac{10}{3(x+f)}$ with $e\neq f$, we obtain
$\tilde{\tilde y}=-\frac{10}{3x}$. To get back to $y=-\frac{8}{3x}$, we use the transformation given by 
$g_2=\bigl(\begin{smallmatrix} 0& b\\ -\frac{1}{b}& 0 \end{smallmatrix}\bigr) \in SL_2(\C)$.
\end{proof}
As a consequence of Proposition \ref{001}, we see that 
the solutions given by $F(x)=(x+B)^{\frac{1}{3}}(x+C)^{\frac{2}{3}}$, $x^{\frac{2}{3}}$ and $x^{\frac{1}{3}}$ 
are equivalent to one another by this $SL_2$ action. 

\begin{proposition}\label{002}
The solutions to (\ref{gc0}) for the parameter $k=\frac{2}{3}$ given by
$y=0$ and $y=-\frac{6}{x}$ are equivalent in the sense of Definition \ref{def1}.
\end{proposition}
\begin{proof}
This is clear from applying $g=\bigl(\begin{smallmatrix} 0 &-1\\1&0\end{smallmatrix}\bigr)$ to the zero solution and its inverse $g^{-1}=\bigl(\begin{smallmatrix} 0 &1\\-1&0\end{smallmatrix}\bigr)$ to $y=-\frac{6}{x}$ to get the zero solution. 
\end{proof}

\section{Schwarz functions and equivalent solutions under $SL_2(\C)$}\label{schwarz-section}

The solutions to equation (\ref{gc0}) can be found using the techniques described in \cite{acht} and \cite{ach}. They can be expressed in terms of logarithmic derivatives involving Schwarz triangle functions. Chazy finds the solutions in \cite{chazy1} and \cite{chazy2} by introducing the auxillary parameter $s$ and treating $x=\frac{z_2(s)}{z_1(s)}$ and taking $y=6 \frac{\der}{\der x}\log z_1(s)$. The  generalised equations are then satisfied if and only if $z_1(s)$, $z_2(s)$ are linearly independent solutions to the hypergeometric differential equation 
\begin{align}\label{hypergeom}
s(1-s) z_{ss}+(c-(a+b+1)s) z_s-a bz=0. 
\end{align} 
Here the subscript denotes differentiation with respect to $s$ and $a$, $b$ and $c$ are constants that depend on $k$.
The solutions can also be written as a closed nonlinear system of first order autonomous differential equations, called the generalised Darboux-Halphen system. From the first order system, the equations can be transformed to a Schwarzian type equation with potential term $V(s)$. The solutions are given precisely by Schwarz triangle functions. Let prime denote differentiation with respect to $x$.
\begin{Definition}
A {\bf Schwarz triangle function} $s(\al,\beta,\ga,x)$ is a solution to the following third order non-linear differential equation \[\{s,x\}+\frac{(s')^2}{2}V(s)=0\] where 
$\{s,x\}=\frac{\der}{\der x}\left(\frac{s''}{s'}\right)-\frac{1}{2}\left(\frac{s''}{s'}\right)^2$ is the Schwarzian derivative and
\[V(s)=\frac{1-
\beta^2}{s^2}+\frac{1-\ga^2}{(s-1)^2}+\frac{\beta^2+\ga^2-\al^2-1}{s(s-1)}\] is the potential. 
\end{Definition}
A Schwarz triangle function determines through its inverse a mapping from the complex upper half plane $\H=\{z\in \C: \Im(z)>0\}$
to a spherical, planar or hyperbolic triangle $\De$ with angles between edges given by $(\al \pi, \beta \pi, \ga \pi)$. The edges of the triangle are given by circular arcs. The inverse map $x:\H \to \De$ is single valued and meromorphic given by
\begin{align}\label{quot}
x(s)=\frac{{}_2F_1(a-c+1,b-c+1;2-c;s)}{{}_2F_1(a,b;c;s)}s^{1-c}.
\end{align}
The image of $0$, $\infty$ and $1$ under $s$ are the vertices of the triangle with one vertex at the origin $x(0)=0$ and another $x(1)$ on the real line (identifying the domain of the triangle as a subset of the complex plane). Here we have 
\begin{align*}
a&=\frac{1}{2}(1-\al-\beta-\ga),\\
b&=\frac{1}{2}(1+\al-\beta-\ga),\\
c&=1-\beta.
\end{align*}

In the solutions to (\ref{gc0}), $x$ is given by the quotient of linearly independent solutions to (\ref{hypergeom}). The general solution to (\ref{hypergeom}) is given by
$
\al z_1(s)+\beta z_2(s)
$
where $z_1$, $z_2$ are linearly independent. We form the quotient $x=\frac{z_2}{z_1}$. If we take a different linear combination instead with 
\begin{align*}
\tilde x=\frac{\beta z_1-\de z_2}{-\al z_1+\ga z_2}=\frac{\beta-\de \frac{z_2}{z_1}}{-\al+\ga \frac{z_2}{z_1}}
=\frac{\beta-\de x}{-\al+\ga x},
\end{align*}
then we find
\begin{align*}
x=\frac{\al \tilde x+\beta}{\ga \tilde x+\de}.
\end{align*}
In other words, if we restrict to $\al$, $\beta$, $\ga$, $\de$ such that $\al \de-\beta \ga=1$, then $x=g\cdot \tilde x$ and $\tilde x=(g^{-1})\cdot x$ for $g\in SL_2$. Hence $SL_2$ equivalent solutions to Chazy's equation are determined by  quotient $\frac{z_2}{z_1}$, and thus are completely determined by the Schwarz function $s$. Every distinct Schwarz function therefore gives rise to $SL_2$ equivalent solutions as in Definition \ref{def1}. We will henceforth just consider the quotient $x=\frac{z_2}{z_1}$ in our computations, modulo constants that agree with the expression (\ref{quot}).

Our goal now is to present the various $(x(s), y(s))$, parametrised by the distinct Schwarz functions $s$ that are found using the general method to solve Chazy's equation \cite{acht}, \cite{ach}. Let us denote
\begin{align*}
\Om_1&=-\frac{1}{2}\frac{\der}{\der x}\log\frac{s'}{s(s-1)},\\
\Om_2&=-\frac{1}{2}\frac{\der}{\der x}\log\frac{s'}{s-1},\\
\Om_3&=-\frac{1}{2}\frac{\der}{\der x}\log\frac{s'}{s}.
\end{align*}
Then $y=-2(\Om_1+\Om_2+\Om_3)$ solves (\ref{gc0}) when $(\al,\beta, \ga)=(\frac{2}{k},\frac{2}{k},\frac{2}{k})$ or $(\frac{2}{k},\frac{1}{3},\frac{1}{3})$ and its cyclic permutations. 
Similarly, we find that
$y=-\Om_1-2\Om_2-3\Om_3$ solves (\ref{gc0}) when $(\al,\beta, \ga)=(\frac{1}{k},\frac{1}{3},\frac{1}{2})$, $(\frac{1}{k},\frac{2}{k},\frac{1}{2})$ or $(\frac{1}{k},\frac{1}{3},\frac{3}{k})$.
Also, $y=-4\Om_1-\Om_2-\Om_3$ solves (\ref{gc0}) when $(\al,\beta, \ga)=(\frac{4}{k},\frac{1}{k},\frac{1}{k})$ or $(\frac{2}{3},\frac{1}{k},\frac{1}{k})$. Each of these values of $(\al, \beta, \ga)$ determine the Schwarz triangle function and the values of $(a,b,c)$ in (\ref{quot}). In the next four sections we compute these for $k=2$, $k=\frac{2}{3}$, $k=3$ and $k=\frac{3}{2}$.

\section{Generalised Chazy equation with $k=2$ and its Schwarz functions }\label{gen2}

In this section we give the general solution to (\ref{gc0}) for $k=2$ and the Schwarz functions that solves the equation in Table \ref{table1}. 

\begin{theorem}
The general solution to the generalized Chazy equation with $k=2$ over the Riemann surface $\P^1=\C\cup\{\infty\}$ is given by 
\begin{align*}
y(x)=-2\left(\frac{1}{x-x_1}+\frac{1}{x-x_2}+\frac{1}{x-x_3}\right)
\end{align*}
and depends on 3 arbitrary points $x_1$, $x_2$, $x_3$ on the Riemann surface. 
\end{theorem}
\begin{proof}
We make the following observation over the complex plane $\C$. Under the substitution $y=\frac{-2 f'}{f}$ for $f$ non-zero, the generalised Chazy equation with $k=2$ is equivalent to the linear 4th order ODE $f''''=0$. The solution to $f''''=0$ is given by the cubic polynomial $f=ax^3+3bx^2+6cx+d$ and therefore the general solution to the $k=2$ Chazy equation over $\C$ is given by its logarithmic derivative
\begin{align*}
y=-\frac{6 a x^2+12 b x+12 c}{a x^3+3b x^2+6c x+d}
\end{align*}
where $a$, $b$, $c$, $d$ are constants of integration. For $a\neq 0$, we can factorize $f=a(x-x_1)(x-x_2)(x-x_3)$ over $\C$, and we obtain
$y=-2\left(\frac{1}{x-x_1}+\frac{1}{x-x_2}+\frac{1}{x-x_3}\right)$.
The general solution for $a \neq 0$ therefore depends on three generic points on $\C$. If we include the point at infinity, we allow solutions with $a=0$ of the form say $y=-2\left(\frac{1}{x-x_1}+\frac{1}{x-x_2}\right)$.
\end{proof}

The above fact has already been observed by Chazy in pages 346-347 of \cite{chazy2}.
Let us tabulate in Table \ref{table1} the Schwarz triangle functions that show up in the solutions to the Chazy equation when $k=2$. 
In the first column, we give the angles $(\al,\beta,\ga)$ of the spherical triangle. 
This determines the values $(a,b,c)=(\frac{1}{2}(1-\al-\beta-\ga),\frac{1}{2}(1+\al-\beta-\ga),1-\beta)$ in the hypergeometric equation (\ref{hypergeom}) and in the second column we give
\begin{align*}
x=\frac{{}_2F_1(a-c+1,b-c+1;2-c;s)}{{}_2F_1(a,b;c;s)}s^{1-c}.
\end{align*} 
In the third column, we invert the second column to present the Schwarz function, with a branch cut chosen for the functions in the last three rows. The entry in the second column of the fifth row requires the solution of a quartic polynomial determined by $x(r)$ on the second column. This same Schwarz function will show up again in the solutions to the $k=\frac{2}{3}$ equation. Finally in the last column we present $y(x)$ as determined by the formulas for $s(x)$ in the fourth column. Here $w=e^{\frac{2\pi i}{3}}$ denotes the cube root of unity. The formulas that give $y$ are discussed and presented in \cite{r16}.  
The series expansion of $x(s)$ around a regular neighbourhood
can be computed and can be checked to see if it agrees with the function presented in Table \ref{table1}.

\begin{table}[h]
 \begin{tabular}{|c|c|c|c|} \hline
$(\al,\beta,\ga)$ & $x(s)$ &$s(x)$ & $y(x)$\\ \hline\hline
$(1,1,1)$ & $\frac{s}{1-s}$ & $\frac{x}{x+1}$ & $-\frac{2}{x}-\frac{2}{x+1}$
 \\ \hline
$(1,\frac{1}{3},\frac{1}{3})$ & $\left(\frac{s}{1-s}\right)^{\frac{1}{3}}$& $\frac{x^3}{x^3+1}$ & $-\frac{2}{x+1}-\frac{2}{x+w}-\frac{2}{x+w^2}$ \\ \hline

$\left(\frac{1}{2 }, \frac{1}{3},\frac{1}{2}\right)$ & $2^{\frac{2}{3}}\left(\frac{1-r}{1+r}\right)^{\frac{1}{3}}$& $\frac{16x^3}{(x^3+4)^2}$ & $-\frac{6 x^2}{x^3+4}$\\
& and $s=1-r^2$ & & 
\\ \hline

$\left(\frac{1}{2},1, \frac{1}{2}\right)$  & $\frac{2}{\sqrt{1-s}}-2$ &  $1-\frac{4}{(x+2)^2}$ & $-\frac{2}{x}-\frac{2}{x+2}-\frac{2}{x+4}$ \\ \hline

$\left(\frac{1}{2},\frac{1}{3}, \frac{3}{2}\right)$ & $\frac{1}{2^{\frac{1}{3}}}\left(\frac{1-r}{1+r}\right)^{\frac{1}{3}}\left(\frac{3+r}{3-r}\right)$& Roots of the quartic polynomial  & $-\frac{3\cdot2^{\frac{1}{3}}}{8}\left(\frac{r+1}{r-1}\right)^{\frac{1}{3}}\frac{r-3}{r^3}(r-1)(r+3)^2$  \\

& and $s=1-r^2$ & $\left(\frac{1}{2}-x^3\right)r^4+(8 x^3+4)r^3$ &$=-\frac{12x^2}{2x^3-1}$\\

& &$+(9-18x^3)r^2+27 x^3-\frac{27}{2}=0$ & 
\\ \hline

$\left(2,\frac{1}{2},\frac{1}{2}\right)$  & $\frac{\sqrt{s(1-s)}}{1-2s}$& $\frac{1}{2}(1\pm \frac{1}{\sqrt{4x^2+1}})$ & $-\frac{2}{x}-\frac{16 x}{4x^2+1}$ \\ \hline

$\left(\frac{2}{3}, \frac{1}{2}, \frac{1}{2}\right)$  & $\frac{3}{2}\frac{\sin(r)}{\cos(r)}$ & $\frac{1}{2}\left(1+\frac{9 (4x^2-3)}{(4x^2+9)^{\frac{3}{2}}}\right)$& $-\frac{2}{x}-\frac{16x}{4x^2-27}$ \\
& and $s=\sin^2(\frac{3}{2}r)$ & & 

\\ \hline
  \end{tabular}
 \caption{Schwarz functions for $k=2$}
\label{table1}
 \end{table}
 
Up to fractional linear transformations, the identity map $x=s$ is given by $s(1,1,1,x)$. This function appears in the first row of Table \ref{table1}. The upper half plane is mapped to the hemisphere of $\S^2$ bound by a great circle and the vertices of the triangle are three points lying on the great circle. Geometrically speaking then the spherical triangles make sense only when $(\al+\beta+\ga)\pi \leq 3\pi$. The occurrence of angles adding up to greater than $3 \pi$ requires us to think of ``triangles" overlapping onto itself (or a branched cover or folded triangle) to realise such exaggeratedly large angles.

The function appearing in the third row of Table \ref{table1} given by $s(\frac{1}{2},\frac{1}{3},\frac{1}{2},x)$ appears in Schwarz's list \cite{schwarz} and has dihedral symmetry. Despite the complicated looking formula for $x$ in the fifth row corresponding to $s(\frac{1}{2},\frac{1}{3},\frac{3}{2},x)$, it is straightforward to determine $y$ in terms of $x$.
\begin{proposition}
Let $s=s(\frac{1}{2},\frac{1}{3},\frac{3}{2},x)$ Then the solution to (\ref{gc0}) with $k=2$ is given by
\begin{align*}
y=\frac{\der}{\der x}\log \frac{(s')^3}{s^2(s-1)^{\frac{3}{2}}}=-\frac{12x^2}{2x^3-1}.
\end{align*}
\end{proposition}
\begin{proof}
From the parametrisation
\begin{align*}
x=\frac{1}{2^{\frac{1}{3}}}\left(\frac{1-r}{1+r}\right)^{\frac{1}{3}}\left(\frac{3+r}{3-r}\right)
\end{align*}
with $s=1-r^2$, we obtain 
\begin{align*}
\der x=-\frac{8}{3}\frac{4^{\frac{1}{3}} r^2}{(r-3)^2(r-1)^{\frac{2}{3}}(r+1)^{\frac{4}{3}}}\der r=\frac{1}{r'}\der r.
\end{align*}
The corresponding formula for $y$ that solves (\ref{gc0}) then gives
\begin{align*}
y=\frac{\der}{\der x}\log\frac{(s')^3}{s^2 (s-1)^{\frac{3}{2}}}=r'\frac{\der}{\der r}\log\frac{(-2r'r)^3}{(1-r^2)^2 (-r^2)^{\frac{3}{2}}}\\
=-\frac{3 \cdot 2^{\frac{1}{3}}}{8}\frac{1}{r^3}(r-3)(r+3)^2(r-1)^{\frac{2}{3}}(r+1)^{\frac{1}{3}}.
\end{align*}
From the formula for $x$, we find
\begin{align*}
\frac{1}{2x^3-1}=\frac{(r-3)^3 (r+1)}{16r^3},
\end{align*}
and therefore
\begin{align*}
y=&-\frac{3}{4}\frac{1}{2^{\frac{2}{3}}}\left(\frac{1-r}{1+r}\right)^{\frac{2}{3}}\left(\frac{3+r}{3-r}\right)^2\frac{(r-3)^3 (r+1)}{r^3}\\
=&-12\frac{1}{2^{\frac{2}{3}}}\left(\frac{1-r}{1+r}\right)^{\frac{2}{3}}\left(\frac{3+r}{3-r}\right)^2\frac{(r-3)^3 (r+1)}{16r^3}\\
=&-\frac{12x^2}{2x^3-1}.
\end{align*}
\end{proof}

The maps between the Schwarz functions for the $k=2$ case are presented in Figure \ref{map1}. We use the notation adopted in \cite{hyper}, where the number over the arrow denotes the degree of the algebraic transformation. These transformations are classified in \cite{hyper}. 

\begin{figure}[h]
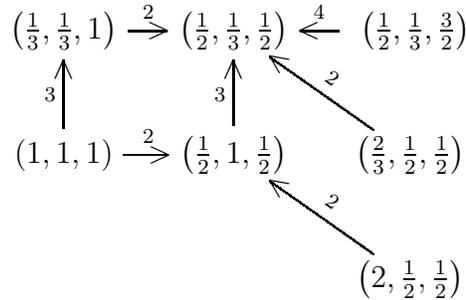

\begin{diagram}
\left(\tfrac{1}{3},\tfrac{1}{3},1\right)  &\rTo^{2} &\left(\tfrac{1}{2},\tfrac{1}{3},\tfrac{1}{2}\right)          &\lTo^{4}   &{\ \left(\tfrac{1}{2},\tfrac{1}{3},\tfrac{3}{2}\right)}\\
\uTo^{3} &           &\uTo^{3} &      \luTo^{2}       &\\
\left(1,1,1\right) &\rTo^{2}     &\left(\tfrac{1}{2},1,\tfrac{1}{2}\right)  &        &\left(\tfrac{2}{3},\tfrac{1}{2},\tfrac{1}{2}\right) \\
&             & &   \luTo^{2}     &    \\
&             &  &               &\left(2,\tfrac{1}{2},\tfrac{1}{2}\right) 
\end{diagram}
\caption{Mapping of Schwarz functions for $k=2$}
\label{map1}
\end{figure}

\section{Generalised Chazy equation with $k=\frac{2}{3}$ and its Schwarz functions }\label{gen23}

In this section we determine the Schwarz functions for the relevant angles $(\al,\beta,\ga)$ that show up in the solutions to the $k=\frac{2}{3}$ equation and present the flat or symmetric $(2,3,5)$-distribution that it determines by computing the anti-derivative
\begin{align*}
F(x)=\int \int e^{\frac{1}{2}\int y(x) \der x} \der x \der x.
\end{align*}
They are presented in Table \ref{table2}. In the table we normalise the constants appearing in $F(x)$ due to integration to be $1$. 

\begin{table}[h]

 \begin{tabular}{|c|c|c|} \hline
$(\al,\beta,\ga)$ & $x(s)$ or $x(r)$ & $F(x(s))$ or $F(x(r))$\\ \hline\hline
$(3,3,3)$ &$\frac{1}{2}\frac{s-2}{2s-1}s^3$ & $s+\frac{1}{2(2s-1)}$
 \\ \hline
$(\frac{1}{3},\frac{1}{3},3)$ & $\frac{1}{2}\left(\frac{s+2}{2s+1}\right)s^{\frac{1}{3}}$ & $\frac{s^{\frac{2}{3}}}{2s+1}$ \\ \hline

$\left(\frac{3}{2 }, \frac{1}{3},\frac{1}{2}\right)$  &  $-\frac{1}{2^{\frac{1}{3}}}\left(\frac{1-r}{1+r}\right)^{\frac{1}{3}}\left(\frac{1+3r}{1-3r}\right)$ &$\frac{(r+1)^{\frac{1}{3}}(r-1)^{\frac{2}{3}}}{3r-1}$ \\ 
& and $s=1-r^2$ & \\
\hline

$\left(\frac{3}{2},3, \frac{1}{2}\right)$  &  $-2\frac{s^2+4s-8}{\sqrt{1-s}}-16$ & $\frac{s-2}{\sqrt{s-1}}$ \\ \hline

$\left(\frac{3}{2},\frac{1}{3}, \frac{9}{2}\right)$  &  
 $-2^{-\frac{4}{3}}\left(\frac{r+3}{r-3}\right)\left(\frac{1-r}{1+r}\right)^{\frac{1}{3}}\frac{3r^4-8r^3+54r^2-81}{3r^4+8r^3+54r^2-81}$ 
 & $\left(\frac{r+1}{r-1}\right)^{\frac{1}{3}}\frac{(r+3)(r^2-9)(r-1)}{3r^4+8r^3+54r^2-81}$ \\ 
 & and $s=1-r^2$ & \\ 
 \hline

$\left(6,\frac{3}{2},\frac{3}{2}\right)$  &  $\frac{(s(1-s))^{\frac{3}{2}}(2s-1)}{128s^4-256s^3+144s^2-16s-1}$ & $\frac{1}{128s^4-256s^3+144s^2-16s-1}$ \\ \hline

$\left(\frac{2}{3}, \frac{3}{2}, \frac{3}{2}\right)$  &  $-\frac{81}{64}\frac{\sin(8r)-2\sin(4r)}{\cos(8r)+2\cos(4r)}$ & $\frac{1}{\cos(8r)+2\cos(4r)}$ \\ 
& and $s=\sin^2(3r)$ & \\
\hline
  \end{tabular}
 \caption{Parametrisations for $k=\frac{2}{3}$}
\label{table2}
 \end{table}
Inverting the function for $s$ in terms of $x$ require the solution of a quartic equation in $s$ except for two cases in the second and third last rows of Table \ref{table2} where the degrees of the polynomials involved are larger. We present $F(x)$ instead of the solutions $y(x)$ because the results take precedence in the theory of $(2,3,5)$-distributions. 

There is a Legendre duality discovered in \cite{annur} that takes equation (\ref{6ode}) to (\ref{noth}) given by the following:
\[
(x,F) \mapsto (t,H)=(F', x F'-F).
\]
We find that $H$ as a function of $t$ solves equation (\ref{noth})
\begin{equation}\label{h-noth}
10(H'')^3H^{(6)}-70(H'')^2H^{(3)}H^{(5)}-49(H'')^2(H^{(4)})^2+
280H''(H^{(3)})^2H^{(4)}-175(H^{(3)})^4=0
\end{equation}
iff equation (\ref{6ode}) holds. The prime in equation (\ref{h-noth}) refers to differentiation with respect to $t$. Equation (\ref{6ode}) can be reduced to the generalised Chazy equation with $k=\frac{2}{3}$ while equation (\ref{h-noth}) can be reduced to the generalised Chazy equation with $k=\frac{3}{2}$. We now provide an example of the computation of the entry in the first row of Table \ref{table2} with the Schwarz function $s(3,3,3,x)$. The rest of the Schwarz functions that show up in Table \ref{table2} can be determined similarly. 

\subsection{The Schwarz function $s(3,3,3,x)$}\label{s333}
As an exercise let us determine the  Schwarz triangle function $s(3,3,3,x)$ and find the flat $(2,3,5)$-distribution that this function determines. See also the first row of Table 1 in \cite{r16}. The solution to $u_{ss}+\frac{1}{4}V(s) u=0$ is
\begin{align*}
u(s)=\al \frac{2s-1}{s(s-1)}+\ga \frac{s^2(s-2)}{s-1}. 
\end{align*}
We take $u_1=\frac{2s-1}{s(s-1)}$, $u_2=\frac{s^2(s-2)}{2(s-1)}$ and obtain
\begin{align}\label{quas}
x(s)=\frac{u_2}{u_1}=\frac{s^3(s-2)}{2(2s-1)}.
\end{align}
This agrees with the expression given by (\ref{quot}).
The Wronskian is given by $c_0=W(u_1,u_2)=(u_2)_su_1-(u_1)_su_2=3$. We find
\begin{align}\label{4yq}
y(x(s))=\frac{u_1^2}{c_0}\frac{\der}{\der s}\log \frac{u_1^6}{c_0^3 s^2 (s-1)^2}=-\frac{4(2s-1)(5s^2-5s+2)}{3 s^3(s-1)^3}.
\end{align}
Using (\ref{quas}) and (\ref{4yq}), we can eliminate $s$ to obtain an equation involving $x$ and $y$.

We can also express $y$ in terms of $x$ directly since (\ref{quas}) says that $s$ is a root of the quartic polynomial 
\begin{align*}
s^4-2s^3-4xs+2x=0, 
\end{align*}
which can always be solved by radicals. 
We solve this quartic polynomial in Appendix \ref{appA}. Substituting this formula for $s=\xi+\frac{1}{2}$ into  (\ref{4yq}) gives the solution to the generalised Chazy equation with $k=\frac{2}{3}$.
The formula for $y$ is really too long and too unsightly to reproduce here.
As a function of $\xi$, we have
\begin{align*}
y(\xi)=-\frac{128}{3}\frac{\xi(20 \xi^2+3) }{(2\xi-1)^3 (2\xi+1)^3}.
\end{align*} 
Alternatively given the polynomial equation of degree 4
\begin{align*}
s^4-2s^3-4xs+2x=0
\end{align*}
and the polynomial equation of degree 6 
\begin{align*}
y s^6-3y s^5+3y s^4-(\frac{40}{3}-y)s^3-20s^2+12s-\frac{8}{3}=0
\end{align*}
that has to be satisfied because of (\ref{4yq}), 
we can eliminate $s$ by forming the resultant polynomial $P(x,y)$ of bidegree $(6,4)$:  
\begin{align*}
P(x,y)=&y^4+\frac{8(4x+1)}{x(2x+1)}y^3+\frac{16}{3}\left(\frac{70 x^2+35 x+4}{x^2(2x+1)^2}\right)y^2\\
&+\frac{256}{27}\left(\frac{(10x+1)(5x+2)(4x+1)}{x^3(2x+1)^3}\right)y+\frac{64}{9}\left(\frac{250x^2+125x+316}{x^3(2x+1)^3}\right)=0. 
\end{align*}
This defines an implicit relation between $x$ and $y$, and the integral curves are solutions to the generalised Chazy equation with parameter $k=\pm \frac{2}{3}$.

To obtain the corresponding flat $(2,3,5)$-distribution $F(x)$, we have to integrate further.
Using (\ref{4yq}) and the fact that
\begin{align*}
\der x=3\frac{s^2(s-1)^2}{(2s-1)^2}\der s,
\end{align*}
we get
\begin{align*}
F''(x(s))&=\frac{(2s-1)^3}{s^4(s-1)^4},\\
F'(x(s))&=3\left(\frac{1}{s}-\frac{1}{s-1}\right),\\
F(x(s))&=\int \int e^{\frac{1}{2}\int y \der x} \der x \der x=-\frac{9}{4}s-\frac{9}{8(2s-1)}.
\end{align*}
Modulo constants of integration, the formula for $F(x(s))$ is just $s+\frac{1}{2(2s-1)}$ as presented in the first row of Table \ref{table2}. Eliminating the parameter $s$ again gives a quartic algebraic curve $P_1(x,F)=0$. For the family of quartic curves parametrised by
\begin{align*}
(x,F)=\left(\frac{s^3(s-2)}{2(2s-1)},c\left(s+\frac{1}{2(2s-1)}\right)\right)
\end{align*}
where $c$ is constant, we can compute the Legendre dual $(x,F) \mapsto (t,H)$ by determining $t=F'$ and $H=x t-F$. This gives
\begin{align*}
(t,H)=\left(\frac{4c}{3}\left(\frac{1}{s-1}-\frac{1}{s}\right),-c\left(\frac{1}{6}+\frac{2}{3}s+\frac{2}{3(s-1)}\right)\right).
\end{align*}

Each row of Table \ref{table2} gives a solution to equation (\ref{6ode}). We have the following. 

\begin{theorem}
Each row of Table \ref{table2} determines a $(2,3,5)$-distribution $D_{F(x)}$ with split $G_2$ symmetry. 
\end{theorem}
We observe that the Schwarz triangle function $s(\frac{1}{2},\frac{1}{3},\frac{3}{2},x)$ appears in both the solutions to the $k=2$ Chazy equation and the $k=\frac{2}{3}$ equation. This gives an intriguing relationship between the solutions to (\ref{gc0}) for $k=\frac{2}{3}$ and $k=2$ determined by this Schwarz function.
\begin{proposition}
Let $s=s(\frac{3}{2},\frac{1}{3},\frac{1}{2},x)$ and suppose $s=1-r^2$ for some $r(x)$. Take the parametrisation for $x(s)$ given in the third row of Table \ref{table2}. We find
\begin{align*}
y_2=\frac{\der}{\der x}\log\frac{(s')^3}{(s-1)^{\frac{5}{2}}s^2}=-\frac{3}{4\cdot 2^{\frac{2}{3}}}\left(\frac{1-r}{1+r}\right)^{\frac{2}{3}}\frac{1}{r}(3r-1)(9r^3+15 r^2+7r+1)
\end{align*}
is a solution to (\ref{gc0}) with $k=2$ while
\begin{align*}
y_{\frac{2}{3}}=\frac{\der}{\der x}\log\frac{(s')^3}{(s-1)^{\frac{3}{2}}s^2}=y_2+\frac{\der}{\der x}\log(s-1)=-\frac{3}{2^{\frac{2}{3}}}\left(\frac{1-r}{1+r}\right)^{\frac{2}{3}}(3r-1)(3r^2+5 r+2)
\end{align*}
is a solution to (\ref{gc0}) with $k=\frac{2}{3}$.
\end{proposition}

The Schwarz functions for the $k=\frac{2}{3}$ equation are determined by the following pull-back maps in Figure \ref{map2}, with the same notation from \cite{hyper}. 

\begin{figure}
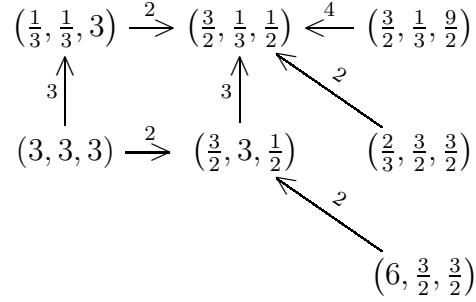

\begin{diagram}[h]
 \left(\tfrac{1}{3},\tfrac{1}{3},3\right)  &\rTo^{2} & \left(\tfrac{3}{2},\tfrac{1}{3},\tfrac{1}{2}\right)         &\lTo^{4}   &{ \left(\tfrac{3}{2},\tfrac{1}{3},\tfrac{9}{2}\right)}\\
\uTo^{3} &           &\uTo^{3} &      \luTo^{2}       &\\
\left(3,3,3\right)&\rTo^{2}     &\ \left(\tfrac{3}{2},3,\tfrac{1}{2}\right) &        &\left(\tfrac{2}{3},\tfrac{3}{2},\tfrac{3}{2}\right)\\
&             & &   \luTo^{2}     &    \\
&             &  &               &\ \left(6,\tfrac{3}{2},\tfrac{3}{2}\right)
\end{diagram}
\caption{Mapping of Schwarz functions for $k=\frac{2}{3}$}
\label{map2}
\end{figure}

The domain of this spherical triangle corresponding to the Schwarz function $s(\frac{1}{2},\frac{1}{3},\frac{3}{2},x)$ has angles $(\frac{1}{2}\pi, \frac{1}{3}\pi, \frac{3}{2}\pi)$. 
This triangle is the complement of the triangle with angles $(\frac{1}{2}\pi, \frac{2}{3}\pi, \frac{1}{2}\pi)$ in a hemisphere with the edge of the hemisphere lying along the $\frac{1}{2}\pi$ and $\frac{2}{3}\pi$ edge of the triangle. Reflecting along the $\frac{3}{2}\pi$, $\frac{1}{2}\pi$ edge gives us the ``triangle" with angles $(\frac{1}{3}\pi, \frac{1}{3}\pi, 3 \pi)$ branched over the vertex with angle $3\pi$. A reflection instead along the equatorial $\frac{1}{3}\pi$, $\frac{1}{2}\pi$ edge gives the triangle with angles $(\frac{2}{3}\pi, \frac{3}{2}\pi, \frac{3}{2}\pi)$, which is also the complement of the $(\frac{4}{3}\pi, \frac{1}{2}\pi, \frac{1}{2}\pi)$ triangle in the sphere.

\section{Generalised Chazy equation with $k=3$ and its Schwarz functions }\label{gen3}
We let $t$ denote now the independent variable. This is to avoid confusion when we compute the Legendre dual curves later on which uses both independent variables $x$ and $t$.   
\begin{theorem}
The general solution to the generalised Chazy equation with $k=3$ over the Riemann surface $\P^1=\C\cup\{\infty\}$ is given by 
\begin{align}\label{tetr}
y(t)=-\frac{3}{2}\left(\frac{1}{t-t_1}+\frac{1}{t-t_2}+\frac{1}{t-t_3}+\frac{1}{t-t_4}\right)
\end{align}
where the $4$ points $t_1$, $t_2$, $t_3$ $t_4$ on the Riemann surface are subject to the constraint $\cQ$ that $12 ae-3b d+c^2=0$ where
\begin{align*}
a(t-t_1)(t-t_2)(t-t_3)(t-t_4)=at^4
+bt^3+c t^2+dt+e.
\end{align*}
\end{theorem}
\begin{proof}
Under the substitution $y=\frac{-3 f'}{2f}$ for $f$ non-zero, the generalised Chazy equation with $k=3$ is equivalent to the nonlinear 4th order ODE $2f''''f-2f'''f'+(f'')^2=0$. Differentiating this ODE gives the linear ODE $f'''''=0$ whose solutions are given by $f=at^4+bt^3+ct^2+dt+e$. Substituting this back into the 4th order ODE, the coefficients are subject to the constraint $12 ae-3b d+c^2=0$ and therefore the general solution to the $k=3$ Chazy equation over $\C$ is given by (\ref{tetr}).
\end{proof}

This general solution has been mentioned in \cite{chazy2} and \cite{cosgrove}. For this parameter, Chazy (\cite{chazy2} p. 347) makes the observation that assuming the four roots are distinct, the roots $t_1$, $t_2$, $t_3$ and $t_4$ can be chosen to lie on the {\it sommets} of the {\it t\'etra\`edre 
r\'egulier}. The condition $\cQ$ defines a projective variety in $\P^3$. 
When one of the points is $t_4=\infty$, the condition on the remaining three roots to be a solution is that they must lie on the conic $\cC=\{t_1,t_2,t_3 \in \C |t_1^2+t_2^2+t_3^2-t_1t_2-t_1t_3-t_2t_3=0\}$. 
In this situation,
\[
y=-2\left(\frac{1}{t-t_1}+\frac{1}{t-t_2}+\frac{1}{t-t_3}\right)
\]
satisfies the $k=2$ equation while
\[
y=-\frac{3}{2}\left(\frac{1}{t-t_1}+\frac{1}{t-t_2}+\frac{1}{t-t_3}\right) 
\]
satisfies the $k=3$ equation. We have

\begin{proposition}
The inclusion of varieties
\begin{diagram}
\cQ        &\subset  &\P^3 \\
\uInto  &           &\uInto\\
\cC         &\subset   &\P^2
\end{diagram}
corresponds to the following isomorphism 
\[
\left\{\begin{tabular}{l}
\mbox{Solutions to $k=3$ equation} \\
\mbox{with one pole at $\{\infty\}$}
\end{tabular}\right\}
\cong
\left\{\begin{tabular}{l}
\mbox{Solutions to $k=2$ equation} \\
\mbox{with poles in $\cC$}
\end{tabular}\right\}
\]
where the isomorphism is given by a constant rescaling. 
\end{proposition}

Determining the Schwarz functions when $k=3$ is a more complicated task than the cases when $k=\frac{2}{3}$ and $k=2$. We present the functions in Table \ref{table3}. Observe that $s(1,\frac{1}{3},\frac{1}{3},t)$ appear both in the $k=2$ and $k=3$ equations. Also observe that for $s(\frac{2}{3},\frac{1}{3},\frac{1}{3},t)$ the formula for $y$ occurs twice, with both 
\begin{align*}
y=\frac{\der}{\der t}\log \frac{(s')^3}{s^2(s-1)^2} \text{ and } y=\frac{\der}{\der t}\log \frac{(s')^3}{s^{\frac{5}{2}}(s-1)^{\frac{5}{2}}} \
\end{align*}
satisfying the $k=3$ Chazy equation. The functions here that appear in Schwarz's list are $s(\frac{1}{3}, \frac{1}{3},\frac{1}{2},t)$ and $s(\frac{2}{3}, \frac{1}{3},\frac{1}{3},t)$, both of which possess tetrahedral symmetry. Again $w=e^{\frac{2 \pi i}{3}}$ denotes the cube root of unity.
\begin{table}[h]
 \begin{tabular}{|c|c|c|c|} \hline
$(\al,\beta,\ga)$ & $t(s)$ & $s(t)$ & $y(t)$\\ \hline\hline
$(\frac{2}{3},\frac{2}{3},\frac{2}{3})$ &$\frac{{}_2F_1(\frac{1}{6},\frac{5}{6};\frac{5}{3};s)}{{}_2F_1(\frac{1}{6},-\frac{1}{2};\frac{1}{3};s)}s^{\frac{2}{3}}$ & $I(-2^{-\frac{2}{3}}w^2t)=\frac{s(t)^2}{4(s(t)-1)}$ & Same as $y(t)$ \\
& &  where $I(t)=s(\frac{2}{3},\frac{1}{3},\frac{1}{2},t)$& for $s(\frac{2}{3},\frac{1}{3},\frac{1}{2},t)$
 \\ \hline
$(\frac{2}{3},\frac{1}{3},\frac{1}{3})$ &  $\frac{{}_2F_1(\frac{1}{6},\frac{5}{6};\frac{4}{3};s)}{{}_2F_1(-\frac{1}{6},\frac{1}{2};\frac{2}{3};s)}s^{\frac{1}{3}}$ & $ \frac{\sqrt{2}}{32}\frac{(\sqrt{t^3+2}+3\sqrt{2})^3 (\sqrt{t^3+2}-\sqrt{2})}{(\sqrt{t^3+2})^3}$ & $-\frac{9}{2}\frac{t^2}{t^3+2}$
\\& & &and $-\frac{6(t^3-4)}{t(t^3-16)}$

\\ \hline

$\left(\frac{1}{3}, \frac{1}{3},\frac{1}{2}\right)$& $\frac{{}_2F_1(\frac{7}{12},\frac{1}{4};\frac{4}{3};s)}{{}_2F_1(-\frac{1}{12},\frac{1}{4};\frac{2}{3};s)}s^{\frac{1}{3}}$ & $-\frac{t^3 (t^3-64)^3}{512(t^3+8)^3}$ & $-\frac{9}{2}\frac{t^2}{t^3+8}$\\ \hline

$\left(\frac{2}{3},\frac{1}{3}, \frac{1}{2}\right)$  &  
$\frac{{}_2F_1(\frac{1}{12},\frac{3}{4};\frac{4}{3};s)}{{}_2F_1(\frac{5}{12},-\frac{1}{4};\frac{2}{3};s)}s^{\frac{1}{3}}$ 
& $\frac{s}{(4s-1)^3}=\frac{1}{512}\frac{t^3(27t^3+64)^3}{(27t^3-8)^3}$  & $-\frac{6}{t}\left(\frac{27t^3+16}{27t^3+64}\right)$\\ \hline

$\left(\frac{1}{3},\frac{1}{3}, 1\right)$  & $s^{\frac{1}{3}}$   & $t^3$ & $-\frac{9}{2}\frac{t^2}{t^3-1}$ \\ 
 \hline

$\left(\frac{4}{3},\frac{1}{3},\frac{1}{3}\right)$  & $\frac{{}_2F_1(\frac{7}{6},-\frac{1}{6};\frac{4}{3};s)}{{}_2F_1(-\frac{1}{2},\frac{5}{6};\frac{2}{3};s)}s^{\frac{1}{3}}$ &  $I(2^\frac{2}{3}t)=4s(t) (1-s(t))$ & Same as $y(t)$\\
& &  where $I(t)=s(\frac{2}{3},\frac{1}{3},\frac{1}{2},t)$&  for $s(\frac{2}{3},\frac{1}{3},\frac{1}{2},t)$

\\ \hline

  \end{tabular}
 \caption{Parametrisations for $k=3$}
\label{table3}
 \end{table}

\begin{proposition}
Let $s=s(\frac{2}{3},\frac{1}{3}, \frac{1}{2},t)$. We have
\begin{align}\label{c1}
 \frac{s}{(4s-1)^3}=\frac{1}{512}\frac{t^3(27t^3+64)^3}{(27t^3-8)^3}
 \end{align}
 as written in the fourth row of Table \ref{table3}.
  \end{proposition}
  \begin{proof}
 This is obtained from considering the pull back map $(\frac{1}{3},\frac{1}{3},\frac{1}{2}) \stackrel{3}\gets (\frac{2}{3},\frac{1}{3},\frac{1}{2})$ in  Figure \ref{map3}. This corresponds to the cubic transformation given in formula (121) of \cite{goursat}. For $\al=-\frac{1}{12}$, we have
 \begin{align*}
 {}_2F_1\left(\frac{5}{12},-\frac{1}{4};\frac{2}{3};s\right)=(1-4s)^{\frac{1}{4}}{}_2F_1\left(-\frac{1}{12},\frac{1}{4};\frac{2}{3};\frac{27s}{(4s-1)^3}\right)
 \end{align*}  
 while for $\al=\frac{1}{4}$, we have
 \begin{align*}
 {}_2F_1\left(\frac{1}{12},\frac{3}{4};\frac{4}{3};s\right)=(1-4s)^{-\frac{3}{4}}{}_2F_1\left(\frac{7}{12},\frac{1}{4};\frac{4}{3};\frac{27s}{(4s-1)^3}\right).
 \end{align*}  
 Together this gives
 \begin{align*}
\frac{{}_2F_1(\frac{1}{12},\frac{3}{4};\frac{4}{3};s)}{{}_2F_1(\frac{5}{12},-\frac{1}{4};\frac{2}{3};s)}s^{\frac{1}{3}}=-\frac{1}{3}\left(\frac{27s}{(4s-1)^3}\right)^{\frac{1}{3}}\frac{{}_2F_1(\frac{7}{12},\frac{1}{4};\frac{4}{3};\frac{27s}{(4s-1)^3})}{{}_2F_1(-\frac{1}{12},\frac{1}{4};\frac{2}{3};\frac{27s}{(4s-1)^3})}.
 \end{align*}
Using the formula for the right hand side as given by the third row of Table \ref{table3} gives 
 \begin{align*}
\frac{s}{(4s-1)^3}=\frac{1}{512}\frac{t^3(27t^3+64)^3}{(27t^3-8)^3}.
\end{align*}
\end{proof} 
 Let us now also explain how we determined the corresponding solution $y(t)$ for $s(\frac{2}{3},\frac{1}{3}, \frac{1}{2},t)$, as this is not so straightforward. 
 \begin{proposition} \label{subst}
 Let $s=s\left(\frac{2}{3},\frac{1}{3}, \frac{1}{2},t\right)$. Then the solution to (\ref{gc0}) with $k=3$ is given by
 \begin{align*}
 y=\frac{\der }{\der t}\log\frac{(s')^3}{s^{\frac{5}{2}} (s-1)^{\frac{3}{2}}}=-\frac{6}{t}\left(\frac{27t^3+16}{27t^3+64}\right).
 \end{align*}
 \end{proposition}
 \begin{proof}
From (\ref{c1}), we obtain
 \begin{align*}
 \frac{s^{\frac{1}{3}}}{4s-1}=\frac{1}{8}\frac{t(27t^3+64)}{(27t^3-8)}.
 \end{align*}
 Let $Q$ denote the right hand side term of the above expression.
 We have
\begin{align*}
-\frac{1}{3}\frac{8s+1}{s^{\frac{2}{3}}(4s-1)^2}\der s=-\frac{1}{3}\frac{8s+1}{s(4s-1)}Q\der s=Q'\der t, 
\end{align*}
or alternatively, 
\begin{align*}
s'=-3\frac{s(4s-1)}{8s+1}\frac{Q'}{Q} .
\end{align*}
The formula for $y$ gives 
\begin{align*}
y=&\frac{\der}{\der t}\log\frac{(s')^3}{s^\frac{5}{2}(s-1)^{\frac{3}{2}}}\\
=&\frac{\der}{\der t}\log\frac{\left(-3\frac{s(4s-1)}{8s+1}\frac{Q'}{Q}\right)^3}{s^\frac{5}{2}(s-1)^{\frac{3}{2}}}\\
=&-3\frac{Q'}{Q}\frac{s(4s-1)}{8s+1}\frac{\der}{\der s}\log\frac{\left(-3\frac{s(4s-1)}{8s+1}\right)^3}{s^\frac{5}{2}(s-1)^{\frac{3}{2}}}+3\frac{\der}{\der t}\log\frac{Q'}{Q}\\
=&\frac{3}{2}\frac{Q'}{Q}\frac{64 s^3-48 s^2+66s-1}{(8s+1)^2 (s-1)}+3\frac{\der}{\der t}\log\frac{Q'}{Q}.
\end{align*}
Therefore we have
\begin{align}\label{form2}
\frac{2Q}{3Q'}\left(y-3\frac{\der}{\der t}\log\frac{Q'}{Q}\right)=\frac{64 s^3-48 s^2+66s-1}{(8s+1)^2 (s-1)}.
\end{align}
Using the fact that
 $\frac{s}{(4s-1)^3}=\frac{1}{512}\frac{t^3(27t^3+64)^3}{(27t^3-8)^3}$, we eliminate $s$ appearing in both this equation and (\ref{form2}) and get
\[
y=-\frac{6}{t}\left(\frac{27t^3+16}{27t^3+64}\right).
\]
\end{proof}
If we make the substitution $t=\frac{-4\si^{\frac{1}{3}}}{\si+2}$
and $s=-\frac{1}{64}\frac{\si(\si+8)^3}{(1-\si)^3}$ into (\ref{c1}), we get an identity. This alternative parametrisation comes from the mapping $(\frac{2}{3},\frac{1}{3},\frac{1}{2}) \stackrel{4}\gets (\frac{2}{3},\frac{1}{3},2)$ in  Figure \ref{map3} and $\si$ agrees with $s(\frac{2}{3},\frac{1}{3},2,t)$ up to a constant reparametrisation of $t$. From this, we find that $(t,y_3)$ given by
\begin{align*}
t&=\frac{\si^{\frac{1}{3}}}{\si+2}\\
y_3&=-\frac{3}{2}\frac{(\si+2)(\si^3+6\si^2-96\si+8)}{\si^{\frac{1}{3}}(\si+8) (\si-1)^2}\\
&=-\frac{3}{2t}\left(1-\frac{8}{\si-1}-\frac{9}{(\si-1)^2}+\frac{8}{\si+8}\right)=-\frac{3}{2t}\left(\frac{1-108t^3}{1-27t^3}\right)
\end{align*}
solves the $k=3$ equation while $(t,y_{\frac{3}{2}})$ given by
\begin{align*}
t&=\frac{\si^{\frac{1}{3}}}{\si+2}\\
y_{\frac{3}{2}}&=-\frac{9}{4}\frac{(\si+2)(\si-10)\si^{\frac{2}{3}}}{(\si-1)^2}\\
&=-\frac{9}{4t}\left(1-\frac{8}{\si-1}-\frac{9}{(\si-1)^2}\right)=\frac{3}{2}y_3+\frac{9}{4t}\frac{8}{\si+8}
\end{align*}
solves the $k=\frac{3}{2}$ equation. The latter solution agrees with the solution given in Section \ref{oct}, which means that we have reparametrised the solutions for $y(t)$ given by $s(\frac{2}{3}, \frac{1}{3},\frac{1}{2},t)$ to those given in the subsequent Section \ref{oct}. 

For brevity we denote $I(t)=s(\frac{2}{3}, \frac{1}{3},\frac{1}{2},t)$,
$J(t)=s(\frac{4}{3}, \frac{1}{3},\frac{1}{3},t)$ and $K(t)=s(\frac{2}{3}, \frac{2}{3},\frac{2}{3},t)$. We have found using quadratic transformations (similar to those given in the next section) that
$I(2^{\frac{2}{3}}t)=4 J(t)(1-J(t))$ and $I(-2^{-\frac{2}{3}}w^2t)=\frac{K(t)^2}{4(K(t)-1)}$ where $w=e^{\frac{2}{3}\pi i}$.

\begin{proposition}\label{jkl}
The functions $I$, $J$, $K$ determine the same solutions to the generalised Chazy equation for $k=3$.
\end{proposition}

\begin{proof}
We have $I(2^{\frac{2}{3}}t)=4 J(t) (1-J(t))$ and $I(-2^{-\frac{2}{3}}w^2t)=\frac{K(t)^2}{4(K(t)-1)}$.
A computation shows that
\begin{align*}
\frac{(I')^3}{I^{\frac{5}{2}}(I-1)^{\frac{3}{2}}}=c_1\frac{(J')^3}{J^{\frac{5}{2}}(J-1)^{\frac{3}{2}}}=c_2\frac{(K')^3}{K^2(K-1)^2}
\end{align*}
for some possible complex constants $c_1$ and $c_2$. By choosing the same logarithmic branch, the logarithmic derivatives of these three functions give the same solution to the $k=3$ Chazy equation. 
\end{proof}

\section{Generalised Chazy equation with $k=\frac{3}{2}$ and its Schwarz functions }\label{gen32}

In this section we present the Schwarz triangle functions that show up in the solutions to the $k=\frac{3}{2}$ equation. 
The Schwarz triangle functions that appear can be grouped into those that have already shown up in the the solutions to the $k=3$ equation and those that have not. The additional ones are presented in Table \ref{table4}. The Schwarz triangle functions $I$, $J$ and $K$ have already shown up in the solutions to the $k=3$ equation and they solve the $k=\frac{3}{2}$ equation (see the previous section) by deforming the $k=3$ solutions along some function of $K$ (or $I$, $J$).

\begin{proposition}
Let $K(t)=s(\frac{2}{3}, \frac{2}{3},\frac{2}{3},t)$. We have
\begin{align*}
\frac{\der}{\der t}\log\frac{(K')^3}{K^{\frac{5}{2}}(K-1)},\hspace{12pt} \frac{\der}{\der t}\log\frac{(K')^3}{K(K-1)^{\frac{5}{2}}} \mbox{~and~} \frac{\der}{\der t}\log\frac{(K')^3}{K^{\frac{5}{2}}(K-1)^{\frac{5}{2}}}
\end{align*}
all satisfying Chazy's equation with $k=\frac{3}{2}$. This comes from considering the different combinations $y=-4 \Om_1-\Om_2-\Om_3$, $y=- \Om_1-4\Om_2-\Om_3$ and $y=- \Om_1-\Om_2-4\Om_3$, discussed in \cite{r16}. In all three cases the function $K$ remains the same since its triangular domain has the property of being equilateral.  
\end{proposition}
Let us denote $y_3=\frac{\der}{\der t}\log\frac{(K')^3}{K^2(K-1)^2}$. This is the solution to the $k=3$ equation in the previous section. 
\begin{corollary}\label{lll}
The functions $y_3+\frac{\der}{\der t}\log \frac{K-1}{\sqrt{K}}$, 
$y_3+\frac{\der}{\der t}\log \frac{K}{\sqrt{K-1}}$ and $y_3+\frac{\der}{\der t}\log \frac{1}{\sqrt{K(K-1)}}$ are all solutions to Chazy's equations with $k=\frac{3}{2}$.
\end{corollary} 
In the above expressions, $K$ can be substituted for $I$ or $J$ as well.

The remaining Schwarz functions that have not shown up in the $k=3$ case are given by $s(\frac{4}{3}, \frac{2}{3},\frac{1}{2},t)$, $s(\frac{4}{3}, \frac{4}{3},\frac{4}{3},t)$, $s(\frac{8}{3}, \frac{2}{3},\frac{2}{3},t)$ and $s(\frac{2}{3},\frac{1}{3},2,t)$.
We discuss $s(\frac{2}{3},\frac{1}{3},2,t)$ in Section \ref{oct}. 
Again for brevity let us denote $L(t)=s(\frac{4}{3}, \frac{2}{3},\frac{1}{2},t)$,
$M(t)=s(\frac{4}{3}, \frac{4}{3},\frac{4}{3},t)$ and $N(t)=s(\frac{8}{3}, \frac{2}{3},\frac{2}{3},t)$. 
The relationship between $L$, $M$ and $N$ are presented in Table \ref{table4}. We determine them in the following fashion, using only quadratic transformations. 
The inversion formula for $L$ gives
\begin{align*}
t(L)=\frac{{}_2F_1(-\frac{1}{12},\frac{5}{4};\frac{5}{3};L)}{{}_2F_1(-\frac{3}{4},\frac{7}{12};\frac{1}{3};L)}L^{\frac{2}{3}}.
\end{align*}
We find
\begin{align*}
\frac{{}_2F_1(-\frac{5}{6},\frac{11}{6};\frac{5}{3};s)}{{}_2F_1(-\frac{3}{2},\frac{7}{6};\frac{1}{3};s)}s^{\frac{2}{3}}=2^{-\frac{4}{3}}\frac{{}_2F_1(-\frac{1}{12},\frac{5}{4};\frac{5}{3};4s(1-s))}{{}_2F_1(-\frac{3}{4},\frac{7}{12};\frac{1}{3};4s(1-s))}(4s(1-s))^{\frac{2}{3}}.
\end{align*}
The left hand side is the inversion formula for $N(t)$, while the right hand side is $2^{-\frac{4}{3}}t(4N(1-N))$ and so
\begin{align*}
L(2^{\frac{4}{3}}t)=4 N(t) (1-N(t)).
\end{align*}

Similarly, we find
\begin{align*}
\frac{{}_2F_1(-\frac{1}{6},\frac{7}{6};\frac{7}{3};s)}{{}_2F_1(-\frac{1}{6},-\frac{3}{2};-\frac{1}{3};s)}s^{\frac{4}{3}}=2^{\frac{4}{3}}w^2\frac{{}_2F_1(-\frac{1}{12},\frac{5}{4};\frac{5}{3};\frac{s^2}{4(s-1)})}{{}_2F_1(-\frac{3}{4},\frac{7}{12};\frac{1}{3};\frac{s^2}{4(s-1)})}\left(\frac{s^2}{4(s-1)}\right)^{\frac{2}{3}},
\end{align*}
where $w=e^{\frac{2}{3}\pi i}$.
The left hand side is the inversion formula for $M(t)$, while the right hand side is $2^{\frac{4}{3}}w^2t(\frac{M^2}{4(M-1)})$ and so
\begin{align*}
L(2^{-\frac{4}{3}}wt)=\frac{M(t)^2}{4(M(t)-1)}.
\end{align*}

\begin{table}[h]

 \begin{tabular}{|c|c|c|} \hline
$(\al,\beta,\ga)$ & $t(s)$ & $s(t)$ \\ \hline\hline
$(\frac{4}{3},\frac{4}{3},\frac{4}{3})$ 
&$\frac{{}_2F_1(-\frac{1}{6},\frac{7}{6};\frac{7}{3};s)}{{}_2F_1(-\frac{1}{6},-\frac{3}{2};-\frac{1}{3};s)}s^{\frac{4}{3}}$& $L(2^{-\frac{4}{3}}wt)=\frac{s(t)^2}{4(s(t)-1)}$ 
 \\ \hline

$\left(\frac{4}{3}, \frac{2}{3}, \frac{1}{2}\right)$  & 
$\frac{{}_2F_1(-\frac{1}{12},\frac{5}{4};\frac{5}{3};s)}{{}_2F_1(-\frac{3}{4},\frac{7}{12};\frac{1}{3};s)}s^{\frac{2}{3}}$ & $L(t)$  \\ \hline

$\left(\frac{8}{3},\frac{2}{3},\frac{2}{3}\right)$  &  $\frac{{}_2F_1(-\frac{5}{6},\frac{11}{6};\frac{5}{3};s)}{{}_2F_1(-\frac{3}{2},\frac{7}{6};\frac{1}{3};s)}s^{\frac{2}{3}}$& $L(2^{\frac{4}{3}}t)=4s(t) (1-s(t))$ 
\\ \hline
$\left(\frac{2}{3},\frac{1}{3},2\right)$  &  $\frac{2s^{\frac{1}{3}}}{s+2}$& Roots of the cubic polynomial \\
& & $s^3+6 s^2+(12-\frac{8}{t^3}) s+8=0$
\\ \hline
  \end{tabular}
 \caption{Parametrisations for $k=\frac{3}{2}$}
\label{table4}
 \end{table}

\begin{proposition}
The functions $L$, $M$, $N$ determine the same solutions to the generalised Chazy equation for $k=\frac{3}{2}$.
\end{proposition}

\begin{proof}
Similar to Proposition \ref{jkl}, a computation shows that
\begin{align*}
\frac{(L')^3}{L^{\frac{5}{2}}(L-1)^{\frac{3}{2}}}=c_3\frac{(M')^3}{M^2(M-1)^2}=c_4\frac{(N')^3}{N^\frac{5}{2}(N-1)^\frac{5}{2}}
\end{align*}
again for some possible complex constants $c_3$ and $c_4$. Again by choosing the same logarithmic branch, the logarithmic derivatives of these three functions give the same solution to the $k=\frac{3}{2}$ Chazy equation. 
\end{proof}

We are now left with the problem of determining what 
\begin{align*}
y(t)=\frac{\der}{\der t}\log \frac{(L')^3}{L^{\frac{5}{2}}(L-1)^{\frac{3}{2}}}
\end{align*}
is. We shall show that this can be reparametrised to give the solution (\ref{lmn}) below.

 \begin{proposition}
 Let $\tau=s(\frac{1}{3},\frac{2}{3},2,\frac{9t}{16})$ and $L=s(\frac{4}{3},\frac{2}{3},\frac{1}{2},t)$. We have 
 \begin{align}\label{ssi}
 \frac{L^{\frac{1}{3}}}{4L-1}=\frac{4}{3}\tau^{\frac{1}{3}}\left(\frac{1-\tau}{1+8\tau}\right).
 \end{align}
 \end{proposition}  
\begin{proof}
We first apply the cubic transformation $(\frac{1}{3},\frac{2}{3},\frac{1}{2}) \stackrel{3}\gets (\frac{4}{3},\frac{2}{3},\frac{1}{2})$. Let $v=\frac{27 L}{(1+8L)^2(1-L)}$. We find $\frac{v}{v-1}=\frac{27L}{(4L-1)^3}$. The formulas (121) of \cite{goursat} give
\begin{align*}
{}_2F_1\left(-\frac{3}{4},\frac{7}{12};\frac{1}{3};L\right)=(1-4L)^{\frac{3}{4}}{}_2F_1\left(-\frac{1}{4},\frac{1}{12};\frac{1}{3};\frac{27L}{(4L-1)^3}\right).
\end{align*}
for $\al=-\frac{1}{4}$ and
\begin{align*}
{}_2F_1\left(-\frac{1}{12},\frac{5}{4};\frac{5}{3};L\right)=(1-4L)^{-\frac{5}{4}}{}_2F_1\left(\frac{5}{12},\frac{3}{4};\frac{5}{3};\frac{27L}{(4L-1)^3}\right).
\end{align*}
for $\al=\frac{5}{12}$. Together this gives the inversion formula for $L$
\begin{align*}
t=\frac{{}_2F_1(-\frac{1}{12},\frac{5}{4};\frac{5}{3};L)}{{}_2F_1(-\frac{3}{4},\frac{7}{12};\frac{1}{3};L)}L^{\frac{2}{3}}
=&\frac{1}{9}\frac{{}_2F_1(\frac{5}{12},\frac{3}{4};\frac{5}{3};\frac{v}{v-1})}{{}_2F_1(-\frac{1}{4},\frac{1}{12};\frac{1}{3};\frac{v}{v-1})}\left(\frac{v}{1-v}\right)^{\frac{2}{3}}.
\end{align*}
Now we apply the degree $4$ transformation $(\frac{1}{3},\frac{2}{3},\frac{1}{2}) \stackrel{4}\gets (\frac{1}{3},\frac{2}{3},2)$.
Let $w=\frac{64\tau(1-\tau)^3}{(1+8\tau)^3}$. We find using equation (127) of \cite{goursat} that the same parameter $\al=-\frac{1}{4}$  gives
\begin{align*}
{}_2F_1\left(-\frac{1}{4},\frac{1}{12};\frac{1}{3};w\right)=(1+8\tau)^{-\frac{3}{4}}{}_2F_1\left(-1,-\frac{2}{3};\frac{1}{3};\tau\right)
\end{align*}
and $\al=\frac{5}{12}$ gives
\begin{align*}
{}_2F_1\left(\frac{3}{4},\frac{5}{12};\frac{5}{3};w\right)=(1+8\tau)^{\frac{5}{4}}{}_2F_1\left(\frac{5}{3},2;\frac{5}{3};\tau\right).
\end{align*}
We therefore obtain
\begin{align*}
\frac{{}_2F_1(\frac{3}{4},\frac{5}{12};\frac{5}{3};w)}{{}_2F_1(-\frac{1}{4},\frac{1}{12};\frac{1}{3};w)}w^{\frac{2}{3}}
=&\frac{{}_2F_1(\frac{5}{3},2;\frac{5}{3};\tau)(1+8\tau)^{\frac{5}{4}} }{{}_2F_1(-1,-\frac{2}{3};\frac{1}{3};\tau)(1+8\tau)^{-\frac{3}{4}} }\left(\frac{4\tau^{\frac{1}{3}}(1-\tau)}{1+8\tau}\right)^{2}\\
=&\frac{16 \tau^{\frac{2}{3}}}{1+2\tau}.
\end{align*}
Equating $w=\frac{v}{v-1}$, we obtain 
\begin{align*}
\frac{64\tau(1-\tau)^3}{(1+8\tau)^3}=\frac{27L}{(4L-1)^3}
\end{align*}
and the parameters are related by
\begin{align*}
\frac{9t}{16}=\frac{\tau^{\frac{2}{3}}}{(1+2\tau)}
\end{align*}
where the right hand side is the inverse $\tau=s(\frac{1}{3},\frac{2}{3},2,\frac{9t}{16})$.
\end{proof}

 Let $R(\tau(t))$ denote the right hand side of (\ref{ssi}).
We have $t=\frac{16\tau^{\frac{2}{3}}}{9(1+2\tau)}$ and
\begin{align*}
L'=-3\frac{L(4L-1)}{8L+1}\frac{R'}{R}.
\end{align*}
Analogous to Proposition \ref{subst}, we find that
\begin{align*}
\frac{2R}{3R'}\left(y-3\frac{\der}{\der t}\log\frac{R'}{R}\right)=\frac{64 L^3-48L^2+66L-1}{(8L+1)^2(L-1)}.
\end{align*}
 Eliminating $L$ between this equation and (\ref{ssi}) give us the solution
\begin{align}\label{lmn}
(t,y)=\left(\frac{16\tau^{\frac{2}{3}}}{9(1+2\tau)}, \frac{81}{64}\frac{(1+2\tau)(10\tau-1)}{\tau^{\frac{2}{3}}(\tau-1)^2}\right).
\end{align} 
Using the formula $H=\int \int e^{\frac{2}{3}\int y\der t}\der t \der t$, we obtain  
\begin{align*}
(t,H)=\left(\frac{16\tau^{\frac{2}{3}}}{9(1+2\tau)}, \frac{1024}{81}\frac{\tau^{\frac{1}{3}}}{1+2\tau}\right).
\end{align*}
The dual curves are parametrised by 
\begin{align*}
(x,F)=\left(\frac{32(4\tau-1)}{9\tau^{\frac{1}{3}}(\tau-1)}, \frac{512}{81}\frac{\tau^{\frac{1}{3}}}{\tau-1}\right).
\end{align*}

\subsection{The Schwarz function $s(\frac{2}{3},\frac{1}{3},2,t)$ }\label{oct}

For the Schwarz function $s(\frac{2}{3},\frac{1}{3},2,t)$, the parametrisation for $t$ is found to be given by
\[
t=\frac{ 2s^{\frac{1}{3}}}{(s+2)}
\]
while the 
formula for $H(t(s))$ is found to be given by
\[
H(t(s))=-\frac{4s^{\frac{2}{3}}}{s+2}.
\]
Let us explain how we determine this. For the values of $(\al,\beta,\ga)=(\frac{2}{3},\frac{1}{3},2)$, we find $(a,b,c)=(-1,-\frac{1}{3},\frac{2}{3})$.
The general solution to $u_{ss}+\frac{1}{4}V(s)u=0$ with these corresponding values
is given by $u=\beta s^{\frac{1}{3}}+\al (s+2)$.
Let $u_1=(s+2)$ and $u_2=2 s^{\frac{1}{3}}$. Then 
\[
t=\frac{2s^{\frac{1}{3}}}{(s+2)}=\frac{s^{\frac{1}{3}}}{{}_2F_1(-1,-\frac{1}{3};\frac{2}{3};s)}
\]
agrees with the formula given by (\ref{quot})
and we find that
\begin{align*}
y=\frac{\der}{\der t}\log\frac{(s')^3}{s^2(s-1)^{\frac{3}{2}}}=-\frac{9}{8}\frac{(s+2)(s-10)}{(s-1)^2}s^{\frac{2}{3}}.
\end{align*}
We thus find that $s(t)$ is given by the root of the cubic equation
\begin{align*}
s^3+6 s^2+\left(12-\frac{8}{t^3}\right)s+8=0.
\end{align*}
Eliminating the variable $s$, we obtain a polynomial algebraic curve $C(t,y)$ given by
\begin{align*}
C(t,y)=&(2t-3)^2 (4t^2+6t+9)^2 y^3+18t(2t-3)(2t^3-27)(4t^2+6t+9)y^2\\
&+324t^2(t^6-45t^3+243)y-1458t^3(5t^3-108)=0.
\end{align*}
We find that
\begin{align*}
\der t=-\frac{4}{3}\frac{s-1}{(s+2)^2 s^{\frac{2}{3}}}\der s
\end{align*}
and therefore
\begin{align*}
H=\int \int e^{\frac{2}{3}\int y \der t}\der t \der t=-\frac{4s^{\frac{2}{3}}}{s+2}.
\end{align*}
We can eliminate $s$ to get
\begin{align*}
H^3-16 t^3-8 H t=0. 
\end{align*}
Using the formula $(x,F)=(H',t H'-H)$, we find that the dual curves are given by
\begin{align*}
(x,F)=\left(-\frac{s-4}{s-1}s^{\frac{1}{3}}, \frac{2 s^{\frac{2}{3}}}{s-1}\right).
\end{align*}

To summarise the results of Section \ref{gen32}, we have the following
\begin{theorem}
The functions $I$, $J$ $K$ give rise to the curve $(t,H)$ parametrised (up to constants) by 
\begin{align}\label{para1}
(t,H)=\left(\frac{2s^{\frac{1}{3}}}{s+2},-\frac{4s^{\frac{2}{3}}}{s+2}\right)
\end{align}
where $s=s(\frac{2}{3},\frac{1}{3},2,t)$
and the dual curve
\begin{align*}
(x,F)=\left(-\frac{s-4}{s-1}s^{\frac{1}{3}}, \frac{2 s^{\frac{2}{3}}}{s-1}\right)
\end{align*}
determines a flat $(2,3,5)$-distribution $\cD_{F(x)}$. 
The functions $L$, $M$ $N$ give rise to the curve $(t,H)$ parametrised (up to constants) by
\begin{align*}
(t,H)=\left(\frac{2s^{\frac{2}{3}}}{1+2s},-\frac{4s^{\frac{1}{3}}}{1+2s}\right)
\end{align*}
where this is obtained by inverting $s \mapsto \frac{1}{s}$ in (\ref{para1}).
The dual curve given by
\begin{align*}
(x,F)=\left(-\frac{4s-1}{s^{\frac{1}{3}}(s-1)}, \frac{2 s^{\frac{1}{3}}}{1-s}\right)
\end{align*}
also determines a flat $(2,3,5)$-distribution $\cD_{F(x)}$. 
\end{theorem}

Figure $\ref{map3}$ shows the transformation maps between the Schwarz triangle functions that show up in the $k=\frac{3}{2}$ and $k=3$ solutions. The Schwarz functions $s(\frac{1}{2},\frac{1}{3},\frac{2}{3},x)$, $s(\frac{2}{3},\frac{2}{3},\frac{2}{3},x)$ and $s(\frac{4}{3},\frac{1}{3},\frac{1}{3},x)$ appear in the solutions to the $k=3$ and $k=\frac{3}{2}$ equations. The three diagrams $\ref{map1}$, $\ref{map2}$ and $\ref{map3}$ can be combined at the nodes labelled by $(\frac{1}{3},\frac{1}{3},1)$ and $\left(\frac{3}{2},\frac{1}{3},\frac{1}{2}\right)$.

\begin{figure}[h]
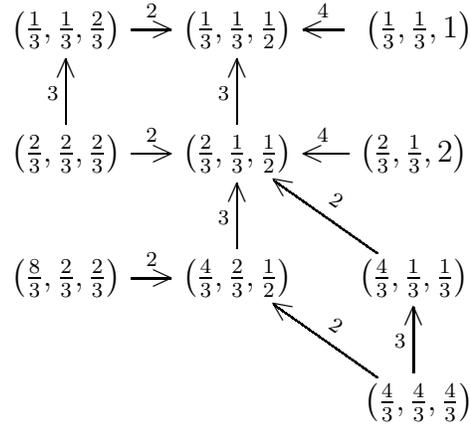

\begin{diagram}
\left(\tfrac{1}{3},\tfrac{1}{3},\tfrac{2}{3}\right)    &\rTo^{2} &\left(\tfrac{1}{3},\tfrac{1}{3},\tfrac{1}{2}\right)         &\lTo^{4}   &{\ \left(\tfrac{1}{3},\tfrac{1}{3},1\right)}\\
\uTo^{3} &           &\uTo^{3} &           &\\
\left(\tfrac{2}{3},\tfrac{2}{3},\tfrac{2}{3}\right)  &\rTo^{2}     &\left(\tfrac{2}{3},\tfrac{1}{3},\tfrac{1}{2}\right)           &\lTo^{4}   &\left(\tfrac{2}{3},\tfrac{1}{3},2\right)  \\
& &\uTo^{3}&\luTo^{2}&\\
\left(\tfrac{8}{3},\tfrac{2}{3},\tfrac{2}{3}\right)  &\rTo^{2}     &\left(\tfrac{4}{3},\tfrac{2}{3},\tfrac{1}{2}\right)           &  &\left(\tfrac{4}{3},\tfrac{1}{3},\tfrac{1}{3}\right)  \\
&     & & \luTo^{2}      & \uTo^{3}\\
&     &       & & {\ \left(\tfrac{4}{3},\tfrac{4}{3},\tfrac{4}{3}\right)}\\
\end{diagram}
\caption{Mapping of Schwarz functions for $k=\frac{3}{2}$ and $k=3$}
\label{map3}
\end{figure}

We end by discussing the shapes of the spherical triangles that show up in the $k=\frac{3}{2}$ and $k=3$ cases. The spherical triangle with angles $(\frac{1}{2}\pi$, $\frac{1}{3}\pi$, $\frac{2}{3}\pi)$ corresponding to $s(\frac{1}{2},\frac{1}{3},\frac{2}{3},x)$ is given by the following. Divide the hemisphere equally into three lunes, with each end having angles $\frac{\pi}{3}$. The domain for the $(\frac{1}{2}\pi, \frac{1}{3}\pi, \frac{2}{3}\pi)$ triangle is the complement of the fundamental domain of the $(\frac{1}{2}\pi, \frac{1}{3}\pi, \frac{1}{3}\pi)$ triangle with tetrahedral symmetry in this lune. Eight of these triangles tile the whole sphere. The triangle with angles $(\frac{2}{3}\pi, \frac{2}{3}\pi, \frac{2}{3}\pi)$ is generated by reflecting this domain along the long edge meeting the right angle, while the triangle with angles $(\frac{1}{3}\pi, \frac{1}{3}\pi, \frac{4}{3}\pi)$ is generated by reflection along the short edge meeting the right angle. 

When $k=\frac{3}{2}$, the triangle with angles $(\frac{1}{2}\pi$, $\frac{2}{3}\pi, \frac{4}{3}\pi)$ also show up. This is the complement of the $(\frac{1}{2}\pi$, $\frac{1}{3}\pi$, $\frac{2}{3}\pi)$ triangle in the hemisphere with the edge between the $\frac{1}{2}\pi$ and $\frac{1}{3}\pi$ angles lying along the equator. The other triangle with angles $(\frac{4}{3}\pi$, $\frac{4}{3}\pi,\frac{4}{3}\pi)$ is generated by reflecting along the long edge meeting the right angle (or equivalently, the equator). This is also the complement of the equilateral triangle with angles $(\frac{2}{3}\pi$, $\frac{2}{3}\pi$, $\frac{2}{3}\pi)$ in the whole sphere. Reflecting along the short edge meeting the right angle gives the ``triangle" with angles $(\frac{8}{3}\pi, \frac{2}{3}\pi, \frac{2}{3}\pi)$, which overlaps at the vertex with angle $\frac{8}{3}\pi$. 

\appendix

\section{Solving quartic equation}\label{appA}
The degree four polynomial equation 
\begin{equation*}
s^4-2s^3-4xs+2x=0
\end{equation*}
can be solved by radicals for $s$ in terms of $x$. 
Substituting $s=\xi+\frac{1}{2}$, we bring it to the form  
\begin{equation}\label{polyt}
\xi^4-\frac{3}{2}\xi^2-(4x+1)\xi-\frac{3}{16}=0. 
\end{equation}
From the coefficients of (\ref{polyt})
we can form the resolvent cubic
\begin{equation}\label{resolvent}
\left(u+\frac{3}{2}\right)\left(u^2+\frac{3}{4}\right)=(4x+1)^2
\end{equation}
and solve for $u$ to obtain
\begin{align*}
u=2(x(2x+1))^{\frac{1}{3}}-\frac{1}{2}, \hspace{12pt}
\left(-\frac{1}{2}+i\frac{\sqrt{3}}{2}\right)2(x(2x+1))^{\frac{1}{3}}-\frac{1}{2}, \hspace{12pt}
\left(-\frac{1}{2}-i\frac{\sqrt{3}}{2}\right)2(x(2x+1))^{\frac{1}{3}}-\frac{1}{2}
\end{align*}
as its roots. 
For $A=\sqrt{u+\frac{3}{2}}$ where $u$ is a solution to (\ref{resolvent}), we form $P=\xi^2+\frac{1}{2}u$ and $
Q=A\xi+\frac{4x+1}{2A}$. Then $P^2=Q^2$ iff (\ref{polyt}) holds, but now $P=\pm Q$ is quadratic in $\xi$ and so the roots can be found using the quadratic formula. For
\[
u_0=(8x(2x+1))^{\frac{1}{3}}-\frac{1}{2},
\]
the roots for $\xi$ are given by: 
\begin{align*}
\xi=&\frac{1}{2\sqrt{(8x(2x+1))^{\frac{1}{3}}+1}}\left((8x(2x+1))^{\frac{1}{3}}+1\right)\\
\pm&\frac{\sqrt{-(8x(2x+1))^{\frac{2}{3}}+(8x(2x+1))^{\frac{1}{3}}+2(4x+1)\sqrt{(8x(2x+1))^{\frac{1}{3}}+1}+2 }}{2\sqrt{(8x(2x+1))^{\frac{1}{3}}+1}},\\
\xi=&-\frac{1}{2\sqrt{(8x(2x+1))^{\frac{1}{3}}+1}}\left(8x(2x+1))^{\frac{1}{3}}+1\right)\\
\pm&\frac{\sqrt{-(8x(2x+1))^{\frac{2}{3}}+(8x(2x+1))^{\frac{1}{3}}-2(4x+1)\sqrt{(8x(2x+1))^{\frac{1}{3}}+1}+2 }}{2\sqrt{(8x(2x+1))^{\frac{1}{3}}+1}}.
\end{align*}

\section{Duality of generalised Chazy equations}\label{uod}

 There is a Legendre duality \cite{annur} that takes the generalised Chazy equation with parameters $k=\frac{2}{3}$ to its dual with parameters $k=\frac{3}{2}$. This is explained in \cite{annur} and \cite{r16}. 
We show that the generalised Chazy equation that is Legendre dual to another generalised Chazy equation has parameter $k$ given either by $\pm\frac{2}{3}$ or $\pm\frac{3}{2}$.
\begin{proposition} Let $m=\frac{4}{36-k^2}$ in equation (\ref{gc0}).
Any generalised Chazy equation can be put into the form 
\begin{align}\label{4thode}
&f^3 f''''-2 (\ell+2) f^2 f' f'''+3(-12 \ell m+\ell-1)f^2 (f'')^2\\
&+12(\ell^2m+6\ell m+1) f (f')^2 f''-(\ell^3 m+12 \ell^2 m+36 \ell m+\ell +6) (f'')^4=0\nonumber
\end{align}
using the substitution $y=\ell\frac{f'}{f}$ for $\ell$ non-zero.
\end{proposition}

To pass to the dual equation, we use the substitution $f=\frac{1}{h}$ and $\frac{\der}{\der x}=\frac{1}{h}\frac{\der }{\der t}$ to determine $f$, $f'$, $f''$, $f'''$ and $f''''$ in terms of $h$ and its derivatives with respect to $t$ and we obtain 
\begin{align*}
&h^3 h''''+(2 \ell-11) h^2 h' h'''+(36 \ell m-3\ell-7)h^2 (h'')^2\\
&+(12\ell^2m-144\ell m-2\ell+59) h (h')^2 h''+(\ell^3 m-24 \ell^2 m+144 \ell m+4\ell-48) (h'')^4=0.
\end{align*}

Hence any equation of the form 
\begin{align*}
&h^3 h''''+(2 j-11) h^2 h' h'''+(36j n-3j-7)h^2 (h'')^2\\
&+(12j^2n-144j n-2j+59) h (h')^2 h''+(j^3 n-24j^2 n+144 jn+4j-48) (h'')^4=0
\end{align*}
is a Chazy equation only if the coefficients agree with that in equation (\ref{4thode}). In this case, $n$ and $m$ is determined completely and is given by either $\frac{16}{135}$ or $\frac{9}{80}$. This gives $k=\pm\frac{2}{3}$ or $k=\pm\frac{3}{2}$.

\section*{Acknowledgements}
The author would like to thank the anonymous reviewers for criticisms of the earlier manuscript.

\end{document}